\documentclass[11pt]{amsart}
\usepackage{graphicx,amsfonts,psfrag,layout,appendix,subfigure}
\usepackage{fancyhdr}

\usepackage{amssymb,amsthm,amsmath, mathabx}
\usepackage{amssymb}
\usepackage{amscd}
\usepackage{amsmath}
\usepackage{amsmath}
 \usepackage[all,cmtip]{xy}
\usepackage{color}

\theoremstyle{plain}

\newtheorem{thm}{Theorem}[section]
\newtheorem{cor}[thm]{Corollary}
\newtheorem{lem}[thm]{Lemma}
\newtheorem{prop}[thm]{Proposition}

\newtheorem{df}[thm]{Definition}

\newtheorem{theorem}{Theorem}[section]
\newtheorem{lemma}[theorem]{Lemma}
\newtheorem{corollary}[theorem]{Corollary}
\newtheorem{proposition}[theorem]{Proposition}

\newtheorem{remark}[theorem]{Remark}

  1

\newcommand{\NN}{\mathbb{N}}

\newcommand{\Q}{\mathbb{Q}}

\newcommand{\Z}{\mathbb{Z}}

\newcommand{\fq}{\mathfrak{q}}

\begin{document}

\title[]{A note on the main conjecture over $\Q$}

\author{Mahesh Kakde, Zdzis\l aw Wojtkowiak  \\ \today}%

\address{Laboratoire Jean Alexandre Dieudonn\'e, U.R.A. au C.N.R.S., No 168, D\'epartment de Math\'ematiques, Universit\'e de Nice-Sophia Antipolis, Parc Valrose -- B.P.No. 71, 06108, Nice Cedex 2, France}
\email{wojtkow@unice.fr}

\address{King's College London,
Department of Mathematics,
London WC2R 2LS,
U.K.}
\email{mahesh.kakde@kcl.ac.uk}


\begin{abstract}
In this note we show how the main conjecture of the Iwasawa theory over $\Q$ has a natural place in the context of the Galois representation of
$Gal (\bar  \Q/\Q)$ on the etale pro-$p$ fundamental group of the projective line minus three points. However we still need to assume the Vandiver 
conjecture to get a proof of the main conjecture in this context.
\end{abstract}
\maketitle

\section{Introduction}

Let us fix an odd prime number $p$. Let $\mathcal{M}_p$ (resp. $M_p$) be the maximal abelian pro-$p$ extension of $\Q(\mu_{p^{\infty}})$ (resp. $\Q(\mu_{p^{\infty}})^+$) unramified outside $p$. Let us denote 
\[
X_{\infty} = Gal(M_p/\Q(\mu_{p^{\infty}})^+)
\]
and
\[
G = Gal(\Q(\mu_{p^{\infty}})^+/\Q).
\]
The group $G$ acts on $X_{\infty}$ by conjugation and hence $X_{\infty}$ is a $\Z_p[[G]]$-module. By a well-known theorem of Iwasawa $X_{\infty}$ is a finitely generated torsion $\Z_p[[G]]$-module. The structure theorem of finitely generated torsion $\Z_p[[G]]$-modules then implies that $X_{\infty}$ is pseudo-isomorphic to $\oplus_i \Z_p[[G]]/(f_i)$ for some $f_i \in Z_p[[G]]$. The main conjecture of Iwasawa theory, proved by Mazur-Wiles \cite{mazurwiles} and independently by Rubin (appendix in \cite{Lang}), states that 
\[
(\prod f_i) = \zeta_{\Q,p} \cdot I(\Z_p[[G]]),
\]
where $I(\Z_p[[G]])$ is the augmentation ideal and $\zeta_{\Q,p} \in Frac(\Z_p[[G]])$ is the $p$-adic zeta function of Kubota-Leopoldt. 
For a more detailed description the reader should consult \cite[pages 1-8]{CS}. 


Let $\mathcal{L}_p$ be the maximal pro-$p$ abelian extension of $\Q (\mu _{p^\infty})$ unramified everywhere. Then $\mathcal{L}_p$ is contained in $\mathcal{M}_p$. Let
\[
 \Gamma :=Gal (\Q (\mu _{p^\infty })/\Q)\cong \Z _p^\times\,.
\]
The group $\Gamma$ acts on $Gal (\mathcal{L}_p/\Q (\mu _{p^\infty}))$ by conjugation.
The main conjecture can also be stated in terms of the $\Z_p[[\Gamma]]$-module $Gal (\mathcal{L}_p/\Q (\mu _{p^\infty}))$ 
(see \cite[the beginning of section 1.4]{CS} and \cite[chapter 14, section 14.5]{Wa}). In this note we present a proof of the main conjecture for 
$Gal (\mathcal{L}_p/\Q (\mu _{p^\infty}))$, assuming the Vandiver conjecture for a prime $p$. We remark that it was already observed by Iwasawa that the main conjecture is a consequence of Vandiver's conjecture. We place our proof in the context of the natural representation of $Gal(\overline \Q /\Q)$ on the pro-$p$ etale fundamental group of $\mathbb{P} ^1_{\overline \Q} -\{0,1,\infty \}$ based at the tangential point $\widearrow{10}$, where the main conjecture over $\Q$ has perhaps its natural place. We point out that the $p$-adic zeta function of Kubota-Leopoldt appears naturally, while considering this Galois representation (see \cite{W2}).

 Our proof is based on construction of the cocycle 
\[
\mathfrak{A}(\widearrow{10}) : G_{\Q} \rightarrow \Z_p[[\Z_p]]
\]
defined in \cite{W2} (and denoted there by $\zeta_p$), whose construction  will be given in section 2 and on the Ihara's formula 
\[
\int_{\Z_p^{\times}} x^{m-1} d\mathfrak{A}(\widearrow{10})([cl](\epsilon))(x) = L_p(m, \omega^{1-m}) (1- p^{m-1}) CW_m(\epsilon),
\]
 for $m>1$ and odd (see \cite[formula on the page 105]{Ihara}).

Let us explain briefly the notations and objects appearing in the formula. Let $\mathcal{U}_\infty ^1$ be the projective limit with respect to the relative norm maps of the principal units of $\Q _p(\mu _{p^n})$. Let $\mathcal{K}_p$ be the field extension of $\Q(\mu _{p^\infty})$ generated by all $p$ powers  roots of $1-\xi _{p^n}^i$ for all $n$ and all $0<i<p^n$. In the formula, $\epsilon \in \mathcal{U}_\infty ^1$, $[cl]:\mathcal{U}_\infty ^1\rightarrow Gal(\mathcal{K}_p/\Q(\mu _{p^\infty})$ is induced from the local class field theory maps, $CW_n:\mathcal{U}_\infty ^1\rightarrow \Z _p$ are Coates-Wiles homomorphisms, $L_p$ is the $p$-adic $L$-function of Kubota-Leopoldt and $\omega$ is the Teichm\"uller character. Our principal result is the following consequence of the Ihara's formula.
\begin{proposition}\label{prop:A}
 Let $\epsilon \in \mathcal{U}_\infty ^1$. Then
 \[
  \mathfrak{A}(\widearrow{10})^\times ([cl](\epsilon))=\zeta _p\cdot \mu ^\times _{\Delta (f_\epsilon )}
 \]
as pseudo-measures on $\mathbb{Z}_p^{\times}$. (Using the well-known isomorphism between the ring of measures on $\mathbb{Z}_p^{\times}$ and the ring of $\Z _p[[\Z _p^\times ]]$ we may consider the above equation as an equality in the total ring of fractions of $\Z_p[[\Z_p^{\times}]]$).
\end{proposition}
In the last formula, $f_\epsilon$ is the Coleman power series associated to $\epsilon$, the power series $\Delta(f_{\epsilon})$ is defined by $(\Delta (f_\epsilon ))(T)=(1+T)f_\epsilon ^\prime (T)/f_\epsilon(T)$. Further, $\mu _g$ denotes the measure on $\Z_p$ associated to a power series $g(T)\in \Z _p[[T]]$ and $\mu_g^\times$ denotes the restriction of the measure on $\Z _p$ to $\Z_p^\times$ (again we have implicitly used isomorphisms between the ring of $\Z_p$-values measures on $\Z_p$, the ring $\Z_p[[\Z_p]]$ and the ring $\Z_p[[T]]$, the latter obtained by fixing $1 \in \Z_p$ as its topological generator). The element $\zeta _p$ is the $p$-adic zeta function we shall construct in this paper.

\begin{corollary}\label{cor:B} Let
 \[
  \mathfrak{A}(\widearrow{10})^\times : Gal(\mathcal{K}_p/\Q(\mu _{p^\infty})) \rightarrow \Z_p[[\Z_p^\times]]
 \]
be induced by $ \mathfrak{A}(\widearrow{10}) : G_{\Q} \rightarrow \Z_p[[\Z_p]] $ by restriction of a measure on $\Z_p$ to $\Z _p^\times$. 
Then we have an isomorphism of $\Z_p[[\Gamma]]$-modules
\[
 Gal(\mathcal{K}_p/\mathcal{K}_p\cap \mathcal{L}_p) \cong (\zeta _p)
\]
where $(\zeta _p)$ is  
the augmentation ideal of $\Z_p[[\Z_p^\times]]$ times  $\zeta _p$.
\end{corollary}

The action of $\Z_p[[\Gamma]]$ on $\Z_p[[\Z_p^\times ]]$ will be defined in section 2. In particular the action of $-1\in  \Gamma$ is an involution, hence any $\Z_p[[\Gamma]]$-module is a direct sum of the $+$ and $-$ parts. 

\begin{lemma}\label{lem:C}
Let us assume the Vandiver conjecture.
Then the map
\[
\mathfrak{A}(\widearrow{10})^\times : Gal(\mathcal{K}_p/\Q(\mu _{p^\infty})) \rightarrow \Z_p[[\Z_p^\times]]^{-}
\]
is an isomorphism.
\end{lemma}

This consequence of the Vandiver conjecture is in fact well known (see \cite[section II, Theorem 6]{Ihara}, \cite[Theorem (7.33)]{Coleman}, \cite[Theorem C]{IK}  and \cite[corollary on page 62]{IS}). The proofs there use  the Ihara power series from \cite{Ihara}. The proof presented in this note is  simpler, and moreover it can be applied in more situations,
for example to study surjectivity of Soul\'e classes for roots of unity or more generally of $l$-adic Galois polylogarithms introduced in \cite{W0}, though in fact it has some common points  with the alternative proof given on page 333 in \cite{IK}.
The proof of the main conjecture over $\Q$ is then an easy consequence of corollary \ref{cor:B} and lemma \ref{lem:C}.
Our last result is the generalization of the Ihara's formula for all integers $ m $. The case $m=1$ is the most interesting as then the residue at $1$ of the $p$-adic zeta function appears.

\section{The cocycle associated to the path from $\widearrow{01}$ to $\widearrow{10}$}

Let $V_n = \mathbb{P}^1_{\overline{\Q}} \setminus (\{0, \infty\} \cup \mu_{p^n})$ and let $\gamma_n$ be a path on $V_n$ from $\widearrow{01}$ to $\frac{1}{p^n} \widearrow{10}$ along the interval $[0,1]$. Let $\pi_1(V_n, \widearrow{01})$ denote the \'etale pro-$p$ fundamental group of $V_n$ based at $\widearrow{01}$. We denote by $x_n$ and by $y_{k,n}$ for $0 \leq k < p^n$ the standard topological generators of $\pi_1(V_n, \widearrow{01})$, loops around $0$ and $\xi_{p^n}^k$ for $0 \leq k < p^n$ respectively (see \cite[section 1]{W1}).

Let $\sigma \in G_{\Q}$. Then, written additively,
\[
\gamma_n^{-1} \sigma(\gamma_n) \equiv \sum_{i=0}^{p^n-1} \alpha_i^{(n)}(\sigma) y_{i,n} \ \text{modulo } (\pi_1(V_n, \widearrow{01}), \pi_1(V_n, \widearrow{01})),
\]
for some coefficients $\alpha_i^{(n)}(\sigma) \in \Z_p$.

\begin{prop}[see \cite{NW}, \cite{W1} and \cite{W2}] For any $\sigma \in G_{\Q}$ the family of functions 
\[
\{  \mathfrak{A}^{(n)}(\widearrow{10})(\sigma) : \Z/p^n\Z \rightarrow \Z_p, i \mapsto \alpha_{i}^{(n)}(\sigma)\}_{n \in \NN}
\]
defines a measure, denoted $\mathfrak{A}(\widearrow{10})(\sigma)$, on $\Z_p$ with values in $\Z_p$.
\label{prop:1.1}
\end{prop}

It follows from the proposition that we get a continuous function 
\[
\mathfrak{A}(\widearrow{10}) : G_{\Q} \rightarrow \Z_p[[\Z_p]].
\]
Let $\tau, \sigma \in G_{\Q}$. It follows from the identity $\gamma_n^{-1}((\tau \sigma)(\gamma_n)) = (\gamma_n^{-1}\tau(\gamma_n)) \cdot \tau(\gamma_n^{-1}\sigma(\gamma_n))$ that 
\begin{equation} \label{eq:1.1}
\alpha_i^{(n)}(\tau \sigma) = \alpha_i^{(n)}(\tau) + \chi(\tau) \alpha_{\chi^{-1}(\tau)i}^{(n)}(\sigma).
\end{equation}
From the last two identities we get 
\begin{equation} \label{eq:1.2}
\alpha_i^{(n)}(\tau \sigma \tau^{-1})= \alpha_i^{(n)}(\tau) - \chi(\sigma) \alpha_{\chi^{-1}(\sigma)i}^{(n)}(\tau) + \chi(\tau) \alpha_{\chi^{-1}(\tau)i}^{(n)}(\sigma). 
\end{equation} 
This identity (\ref{eq:1.2}) motivates the next definition.

\begin{df} Define the action of $\Gamma = \Z_p^{\times}$, hence the action of $G_{\Q}$ via the $p$-cyclotomic character, on $\Z_p[[\Z_p]]$ 
by the formula
\begin{equation} \label{eq:1.3}
c(\sum_{i=1}^m a_i[x_i]) := \sum_{i=1}^m ca_i [cx_i]
\end{equation}
for elements of $\Z_p[\Z_p]$ and we extend by continuity to the action on $\Z_p[[\Z_p]]$. 
\end{df}

It follows immediately from the identity (\ref{eq:1.1}) that the function 
\[
\mathfrak{A}(\widearrow{10}) :G_{\Q} \rightarrow \Z_p[[\Z_p]]
\]
is a cocycle.

Let $a$ belong to a number field $K$ contained in $\overline \Q$. We denote by 
\[
 \kappa (a):G_{K(\mu _{p^\infty})}\rightarrow \Z_p
\]
the Kummer character associated to $a$.

\begin{prop}[see Proposition 2 in \cite{W2}] \label{prop:1.2} Let $\sigma \in G_{\Q(\mu_{p^{\infty}})}$. Then 
\[
\mathfrak{A}^{(n)}(\widearrow{10})(\sigma) = \kappa(\frac{1}{p^n})(\sigma)[0] + \sum_{i=1}^{p^n-1} \kappa(1-\xi_{p^n}^i)(\sigma)[i].
\]
\end{prop}

We recall from the introduction that 
\[
\mathcal{K}_p:= \Q_p(\mu_{p^{\infty}})((1-\xi_{p^n}^i)^{1/p^n} : 0 < i < p^n, n,m \in \NN).
\]
Observe that $\mathcal{K}_p$ is an abelian, unramified outside $p$, pro-$p$ extension of $\Q(\mu_{p^{\infty}})$. 
Let $\tau \in Gal(\Q(\mu_{p^{\infty}})/\Q)$ and $\overline{\tau}$ be a lifting of $\tau$ to $Gal(\mathcal{K}_p/\Q)$. 
The group $\Gamma =Gal(\Q(\mu_{p^{\infty}})/\Q) \cong \Z_p^{\times}$ acts on $Gal(\mathcal{K}_p/\Q(\mu_{p^{\infty}}))$ by the standard formula 
\[
\tau(\sigma) = \overline{\tau} \sigma \overline{\tau}^{-1}.
\]
It follows from proposition \ref{prop:1.2} that the map $\mathfrak{A}(\widearrow{10}) : G_{\Q} \rightarrow \Z_p[[\Z_p]]$ factors through the Galois group $Gal(\mathcal{K}_p/\Q)$. We denote the restriction of $\mathfrak{A}(\widearrow{10}) : Gal(\mathcal{K}_p/\Q) \rightarrow \Z_p[[\Z_p]]$ to the subgroup $Gal(\mathcal{K}_p/\Q(\mu_{p^{\infty}}))$ also by $\mathfrak{A}(\widearrow{10})$. 

\begin{prop} \label{prop:1.3} The map 
\[
\mathfrak{A}(\widearrow{10}) : Gal(\mathcal{K}_p/\Q(\mu_{p^{\infty}})) \rightarrow \Z_p[[\Z_p]]
\]
is a continuous injective homomorphism of $\Z_p[[\Gamma]]$-modules. 
\end{prop}

\begin{proof} The formula (\ref{eq:1.1}) implies that $\mathfrak{A}(\widearrow{10})$ is a morphism of $\mathbb{Z}_p$ modules. 
It follows from equation (\ref{eq:1.2}) that that for $\sigma \in G_{\mathbb{Q}(\mu_{p^\infty})}$ we have 
\[
\alpha_i^{(n)}(\tau \sigma \tau^{-1}) = \chi_{(\tau)} \alpha_{\chi^{-1}(\tau)i}^{(n)}(\sigma)
\]
Therefore 
\[
\sum_{i=0}^{p^n-1} \alpha_i^{(n)}(\tau \sigma \tau^{-1})[i] = \sum_{i=0}^{p^n-1}\chi(\tau) \alpha_{\chi(\tau)^{-1} i}^{(n)}(\sigma)[i] = \chi(\tau) \sum_{j=0}^{p^n-1} \alpha_j^{(n)}(\sigma) [\chi(\tau) j].
\]
Hence it follows that $\mathfrak{A}(\widearrow{10})$ is a $\Gamma$-map. The injectivity following from the definition of the field $\mathcal{K}_p$ and the explicit description of $\mathfrak{A}(\widearrow{10})$ given in proposition \ref{prop:1.2}. 
\end{proof}

\section{The restriction to $\Z_p^{\times}$} 

The multiplicative group $\Z_p^{\times}$ is a closed and open subset of $\Z_p$. It follows from the formula in equation (\ref{eq:1.3}) that the restriction map 
\[
\mathcal{R}: \Z_p[[\Z_p]] \rightarrow \Z_p[[\Z_p^{\times}]],
\]
which associates to a measure $\mu$ on $\Z_p$ its restriction to $\Z_p^{\times}$, is a morphism of $\Z_p[[\Gamma]]$-modules. We shall also denote the measure $\mathcal{R}(\mu)$ by $\mu^{\times}$. 
For $\sigma \in Gal (\overline \Q/\Q(\mu _{p^\infty }))$, we denote the restriction of the measure $\mathfrak{A}(\widearrow{10})(\sigma )$ to $\Z_p^\times$, by $\mathfrak{A}(\widearrow{10})^\times(\sigma )$. We denote the composition 
\[
z_p : Gal(\mathcal{K}_p/\Q(\mu_{p^{\infty}})) \xrightarrow{\mathfrak{A}(\widearrow{10})} \Z_p[[\Z_p]] \xrightarrow{\mathcal{R}} \Z_p[[\Z_p^{\times}]]
\]
by $z_p$. The action of $-1 \in \Gamma$ is an involution, hence $Gal(\mathcal{K}_p/\Q(\mu_{p^{\infty}}))$ and $\Z_p[[\Z_p^{\times}]]$ are the direct sums of their $+$ and $-$ parts. 
Moreover, $z_p$ induces morphism of $\Z_p[[\Gamma]]$-modules 
\[
z_p^{\epsilon} : Gal(\mathcal{K}_p/\Q(\mu_{p^{\infty}}))^{\epsilon} \rightarrow \Z_p[[\Z_p^{\times}]]^{\epsilon},
\]
for $\epsilon \in \{+, -\}$. 

\begin{prop}[See also \cite{W2}, corollary 1] \label{prop:3.1} We have 
\begin{itemize}
\item[(i)] The morphisms $z_p$ and $z_p^{-}$ are injective.
\item[(ii)] The group $Gal(\mathcal{K}_p/\Q(\mu_{p^{\infty}}))^+  = 0$. 
\end{itemize}

\end{prop}

\begin{proof} Observe that $\kappa(1- \xi_{p^n}^i)$, for $0 < i < p^n$ and $(i,p) >1$, as well as $\kappa(\frac{1}{p^n})$ 
can be expressed as sums of $\kappa(1-\xi_{p^n}^k)$ with $0 < k < p^n$ and $(k,p)=1$. 
Hence it follows from propositions \ref{prop:1.2} and \ref{prop:1.3} that $z_p$ is injective. 

Let $a = (\sum_{i=0}^{p^n-1} a_i^{(n)}[i])_{n \in \NN} \in \Z_p[[\Z_p]]$. Then $(-1) \in \Gamma$ acts on $a$ as follows
\[
(-1)a = \left( -a_0^{(n)}[0] - \sum_{i=1}^{p^n-1} a_i^{[n]}[p^n-i] \right)_{n \in \mathbb{N}}.
\]
Hence it follows that 
\begin{equation} \label{eq:2.1}
a^+  = \frac{1}{2} \left( \sum_{i=1}^{p^n-1} (a_i^{(n)} - a_{p^n-i}^{(n)})[i] \right)_{n \in \NN}
\end{equation}
and 
\begin{equation} \label{eq:2.2}
a^- = \frac{1}{2} \left(2 a_0^{(n)}[0] + \sum_{i=1}^{p^n-1}(a_i^{(n)} + a_{p^n-i}^{(n)})[i] \right)_{n \in \NN}.
\end{equation}
For $\sigma \in Gal(\mathcal{K}_p/\Q(\mu_{p^{\infty}}))$, we have 
\[
\mathfrak{A}(\widearrow{10})(\sigma) = \left( \kappa\left(\frac{1}{p^n}\right)(\sigma)[0]+\sum_{i=1}^{p^n-1} \kappa(1-\xi_{p^n}^{-i})(\sigma)[i]) \right)_{n \in \NN}
\]
by proposition \ref{prop:1.2}. The identity $1 - \xi_{p^n}^{-i} = (-\xi_{p^n}^{-i})(1-\xi_{p^n}^i)$ implies that 
$\kappa(1- \xi_{p^n}^{-i}) = \kappa(1-\xi_{p^n}^i)$ on $Gal(\mathcal{K}_p/\Q(\mu_{p^{\infty}}))$. 
Hence it follows that for $\sigma \in Gal(\mathcal{K}_p/\Q(\mu_{p^{\infty}}))$, we have $\mathfrak{A}(\widearrow{10})(\sigma)^+ = 0$ and
\[
\mathfrak{A}(\widearrow{10})(\sigma)^- = \left(\kappa(\frac{1}{p^n})(\sigma)[0]+ \sum_{i=1}^{p^n-1}  \kappa(1- \xi_{p^n}^i)(\sigma)[i]\right)_{n \in \NN}.
\]
Therefore
\[
Gal(\mathcal{K}_p/\Q(\mu_{p^{\infty}}))^+ = 0
\]
and the map $z_p^-$ given by 
\[
z_p^-(\sigma) = \left(\sum_{i=1, (i,p)=1}^{p^n-1}  \kappa(1- \xi_{p^n}^i)(\sigma)[i]\right)_{n \in \NN}
\]
is injective.
\end{proof}

Let $\mathcal{E}_{n}$ be the group of $p$-units in $\mathbb{Q}(\mu_{p^n})$, i.e. $\mathcal{E}_{n} = \mathbb{Z}[\mu_{p^n}]\left[\frac{1}{p}\right]^{\times}$ and let $\mathcal{C}_n$ be the group of cyclotomic $p$-units of $\mathbb{Q}(\mu_{p^n})$, i.e. the subgroup of $\mathcal{E}_n$ generated by $\mu_{p^n}$ and elements $1 - \xi_{p^n}^k$ for $0 < k < p^n$. Let $S_n$ be a set of all integers between $0$ and $p^n/2$ that are coprime to $p$.

\begin{lem} \label{lem:3.2} The following conditions are equivalent:
\begin{itemize}
\item[(i)] the map
\[
 \alpha _n:Gal(\mathcal{K}_p/\mathbb{Q}(\mu_{p^n}))\rightarrow \oplus_{i\in S_n} \mathbb{Z}/p^n\mathbb{Z}
\]
given by 
\[
\sigma \mapsto (\kappa(1- \xi_{p^n}^i)(\sigma))_{i\in S_n},
\]
is surjective;
\item[(ii)] $p$ does not divide $|\mathcal{E}_n/\mathcal{C}_n|$.
\end{itemize}
\end{lem}
\begin{proof} The map
\[
 \mathcal{E}_n/\mathcal{E}_n^{p^n}\rightarrow \mathbb{Q}(\mu_{p^n})^\times/{\mathbb{Q}(\mu_{p^n})^\times }^{p^n}
\]
induced by the inclusion $\mathcal{E}_n \subset \mathbb{Q}(\mu_{p^n})^\times$ is injective. Therefore by the Kummer theory the map $\alpha _n$ is surjective if and only if the subgroup of $\mathcal{E}_n/\mathcal{E}_n^{p^n}$ generated by elements $1-\xi_{p^n}^i$, $i\in S_n$ is isomorphic to $(\Z/p^n\Z)^{|S_n|}$. It follows from  \cite[theorem 8.9]{Wa} that there is no multiplicative relation between $\xi_{p^n}$ and the elements $1-\xi_{p^n}^i$  for $i\in S_n$. Hence $\alpha_n$ is surjective if and only if the subgroup of $\mathcal{E}_n/\mathcal{E}_n^{p^n}$ generated by elements $1-\xi_{p^n}^i$, $i\in S_n$ and by $\xi_{p^n}$ is isomorphic to $(\Z/p^n\Z)^{|S_n|+1}$. The last condition holds if and only if the subgroup of $\mathcal{E}_n/\mathcal{E}_n^{p}$ generated by elements $1-\xi_{p^n}^i$, $i\in S_n$ and by $\xi_{p^n}$ is isomorphic to $(\Z/p\Z)^{|S_n|+1}$.

If $p$ divides $|\mathcal{E}_n/\mathcal{C}_n|$ then there is $u\in \mathcal{E}_n$ such that $u^p\in \mathcal{C}_n$ but $u\notin \mathcal{C}_n$.
Any element of $ \mathcal{C}_n$ can be written in a unique way in the form 
$\pm \xi _{p^n}^{a_0}\prod_{i\in S_n}(1-\xi _{p^n}^i)^{a_i}$ (see \cite[theorem 8.9]{Wa}).
Hence
\[
 u^p=\pm \xi _{p^n}^{a_0}\prod_{i\in S_n}(1-\xi _{p^n}^i)^{a_i}\,,
\]
where at least one of $a_i\nequiv 0$ modulo $p$. But then the subgroup of $\mathcal{E}_n/\mathcal{E}_n^{p}$ generated by $\xi _{p^n}$ and   
$1-\xi_{p^n}^i$ for $i\in S_n$ will have the rank smaller than $|S_n|+1$. Therefore i) implies ii).

If $p$ does not divide $|\mathcal{E}_n/\mathcal{C}_n|$ then the map $\mathcal{C}_n/\mathcal{C}_n^p\rightarrow \mathcal{E}_n/\mathcal{E}_n^{p}$
is injective. The elements $\xi _{p^n}$ and $1-\xi_{p^n}^i$ for $i\in S_n$ taken modulo $\mathcal{C}_n^p$ form a basis of 
$\mathcal{C}_n/\mathcal{C}_n^p$. But it implies, going backwards all equivalent statements, that i) holds. 
\end{proof}

\begin{lem} \label{lem:3.3} The following conditions are equivalent:
\begin{itemize}
\item[(i)] the map 
\[
z_p^{-}: Gal(\mathcal{K}_p/\mathbb{Q}(\mu_{p^{\infty}})) = Gal(\mathcal{K}_p/\mathbb{Q}(\mu_{p^{\infty}}))^- \rightarrow \mathbb{Z}_p[[\mathbb{Z}_p^{\times}]]^-
\]
is an isomorphism of $\mathbb{Z}_p[[\Gamma]]$-modules;
\item[(ii)] $p$ does not divide $|\mathcal{E}_n/\mathcal{C}_n|$ for all $n \geq 1$.
\end{itemize}
\end{lem}
\begin{proof} If  $p$ does not divide $|\mathcal{E}_n/\mathcal{C}_n|$ then the map
\[
 \alpha _n:Gal(\mathcal{K}_p/\mathbb{Q}(\mu_{p^n}))\rightarrow \oplus_{i\in S_n} \mathbb{Z}/p^n\mathbb{Z}
\]
given by 
\[
\sigma \mapsto (\kappa(1- \xi_{p^n}^i)(\sigma))_{i\in S_n},
\] 
is surjective by lemma 3.2. Let $m\geq n$. Let us assume that $p$ does not divide $|\mathcal{E}_m/\mathcal{C}_m|$. 
Then it follows from the Bass theorem (see \cite[page 150 and Theorem 8.9.]{Wa}) that the map
\[
 \alpha _n^{(m)}:Gal(\mathcal{K}_p/\mathbb{Q}(\mu_{p^m}))\rightarrow \oplus_{i\in S_n} \mathbb{Z}/p^m\mathbb{Z}
\]
given by 
\[
\sigma \mapsto (\kappa(1- \xi_{p^n}^i)(\sigma))_{i\in S_n},
\] 
is surjective. Hence the assumption ii) implies that for all $n\in \NN$ the maps
\[
 \alpha _n^\infty:Gal(\mathcal{K}_p/\mathbb{Q}(\mu_{p^\infty}))\rightarrow \oplus_{i\in S_n} \mathbb{Z}_p
\]
given by 
\[
\sigma \mapsto (\kappa(1- \xi_{p^n}^i)(\sigma))_{i\in S_n},
\] 
are surjective. Therefore the map
\[
z_p^{-}: Gal(\mathcal{K}_p/\mathbb{Q}(\mu_{p^{\infty}}))  \rightarrow \mathbb{Z}_p[[\mathbb{Z}_p^{\times}]]^-
\] 
is surjective. Hence it follows from proposition 3.1 that the map $z_p^{-}$ is an isomorphism.

It follows from lemma 3.2 that i) implies ii).
\end{proof}

\section{The multiplicative structure of $\Z_p[[\Z_p^{\times}]]$} 

We denote by $\Z_p(k)$ the ring $\Z_p$ equipped with the action of $\Gamma$ given by 
\[
x(a):= x^ka.
\]
The infinite product $\prod_{k=1}^{\infty} \Z_p(k)$ is equipped with a multiplicative structure given by 
\[
(a_k)_{k=1}^{\infty} \cdot (b_k)_{k=1}^{\infty} = (a_kb_k)_{k=1}^{\infty}.
\]
The $\Z_p[[\Gamma]]$-module $\Z_p[[\Z_p^{\times}]]$ has a multiplication induced by a group multiplication in $\Z_p^{\times}$. Let us define a map 
\[
\Phi: \Z_p[[\Z_p^{\times}]] \rightarrow \prod_{k=1}^{\infty} \Z_p(k)
\]
by the formula
\[
\Phi(\mu) = \left(\int_{\Z_p^{\times}} x^{k-1}d\mu(x)\right)_{k}.
\]

\begin{lemma} \label{lem:3.1} The map $\Phi$ is a continuous injective morphism of $\Z_p[[\Gamma]]$-modules. Moreover, $\Phi(\mu \cdot \nu) = \Phi(\mu)\Phi(\nu)$ for any $\mu, \nu \in \Z_p[[\Z_p^{\times}]]$. 
\end{lemma}

\begin{proof} It is evident that $\Phi$ is continuous. Let us first show that the map $\Phi$ is a map of $\Z_p[[\Gamma]]$. 
By continuity of the it is enough to check that $\Phi(g\mu) = g \Phi(\mu)$ for each 
$g \in \Gamma$ and for $\mu = \sum_{x \in \Z_p^{\times}} a_x[x] \in \Z_p[\Z_p^{\times}]$. By our convention (see formula (\ref{eq:1.3})) 
\[
g \mu = \sum_{x}  ga_x [gx].
\]
Hence 
\[
\Phi(g\mu) = \left(\sum_{x} g^kx^{k-1} a_x\right)_{k} = g \cdot \left(\sum_{x} x^{k-1}a_x\right)_k = g \Phi(\mu).
\]
Next we show that $\Phi$ is multiplicative i.e. $\Phi(\mu\nu) = \Phi(\mu) \Phi(\nu)$. Again by continuity we can need to check this for $\mu = \sum_x a_x[x]$ and $\nu = \sum_x b_x[x]$ in $\Z_p[\Z_p^{\times}]$. We have
\begin{align*}
\Phi(\mu) \Phi(\nu) &= \left(\sum_x a_x x^{k-1}\right)_k\left(\sum_{y} b_y y^{k-1}\right)_k \\
&= \left(\sum_{z} \left(\sum_{x,y: xy = z} a_xb_y\right) z^{k-1}\right)_k \\
&= \Phi(\mu\nu).
\end{align*}

\end{proof}

Observe that 
\[
(\prod_{k=1}^{\infty} \Z_p(k))^+ = \prod_{k=1}^{\infty} \Z_p(2k)
\]
and 
\[
(\prod_{k=1}^{\infty} \Z_p(k))^- = \prod_{k=1}^{\infty} \Z_p(2k-1).
\]
It follows from proposition \ref{prop:3.1} and lemma \ref{lem:3.1} that the composition 
\[
Gal(\mathcal{K}_p/\Q(\mu_{p^{\infty}}))^- \xrightarrow{z_p^-} \Z_p[[\Z_p^{\times}]]^- \xrightarrow{\Phi^{-}} \prod_{k=1}^{\infty} \Z_p(2k-1)
\]
is an injective  morphism of $\Z_p[[\Gamma]]$-modules. 

\section{The Coates-Wiles homomorphism}

We recall briefly some results of Coleman presented in \cite{CS}. Let $\mathcal{U}_n := \Z_p[\mu_{p^{n+1}}]^{\times}$ and put $\mathcal{U}_{\infty}:= \varprojlim_n \mathcal{U}_n$, with the projective limit taken with respect to the norm maps. Let $\mathcal{N}: \Z_p[[T]] \rightarrow \Z_p[[T]]$ be the norm map defined in \cite{CS}. Put $W:= \{f \in \Z_p[[T]]^{\times} : \mathcal{N}(f) = f\}$. For each integer $n \geq 0$, we fix a $p^{n+1}$ root $\xi_{p^{n+1}}$ of 1 such that $\xi_{p^{n+1}}^p = \xi_{p^n}$. 

\begin{theorem}[See \cite{CS}, theorem 2.1.2 and corollary 2.3.7] \label{thm:4.1} Let $\epsilon = (\epsilon_n) \in \mathcal{U}_{\infty}$, then there exists a unique $f_{\epsilon}(T) \in \Z_p[[T]]^{\times}$ such that $f(\xi_{p^{n+1}}-1) = \epsilon_n$ for all $n \geq 0$. Moreover, the map $\epsilon \mapsto f_{\epsilon}(T)$ defines an isomorphism of groups from $\mathcal{U}_{\infty}$ to $W$. 
\end{theorem}

Let us define $\Delta: \Z_p[[T]]^{\times} \rightarrow \Z_p[[T]]$ by
\[
\Delta(f) = (1+T)f'(T)/f(T).
\] 
For $f \in \Z_p[[T]]$ we define 
\[
\varphi(f)(T) := f((1+T)^p-1).
\]
The trace operator $\psi: \Z_p[[T]] \rightarrow \Z_p[[T]]$ is a continuous $\Z_p$-homomorphism characterized by the identity
\[
(\varphi \circ \psi)(f)(T) = \frac{1}{p} \sum_{\xi \in \mu_{p-1}} f(\xi(1+T)-1).
\]
We set 
\[
\Z_p[[T]]^{\psi = id} = \{f \in \Z_p[[T]] : \psi(f) =f \}.
\]
We define a map 
\[
[Col] : \mathcal{U}_{\infty} \rightarrow \Z_p[[T]]
\]
to be the map $\epsilon \mapsto \Delta(f_{\epsilon}(T))$. 

\begin{corollary} \label{cor:4.2} The map $[Col]$ induces a surjective homomorphism of $\Z_p[[\Gamma]]$-modules 
\[
[Col] : \mathcal{U}_{\infty} \rightarrow \Z_p[[T]]^{\psi = id}
\]
with $ker([Col]) = \mu_{p-1}$. 
\end{corollary}
\begin{proof} It follows from \cite[Theorem 2.4.6 and lemma 2.4.5]{CS} that $[Col]$ is surjective with $ker([Col]) = \mu_{p-1}$. 

We recall that the action of $\Gamma$ on $\Z_p[[T]]$ is defined by $c(f) = cf((1+T)^c-1)$. The map $\epsilon \mapsto f_{\epsilon}(T)$ commutes with 
the action of $c \in \Gamma$ defined by $c(f) = f((1+T)^c-1)$. Hence, when we compose with $\Delta$, we get that $[Col]$ is a $\Gamma$-map. 
The map $[Col]$ is a morphism of $\Z$ modules, hence by continuity it is a $\Z_p$, whence a map of $\Z_p[\Gamma]$-modules. 
\end{proof}

We recall the definition of the Coates-Wiles homomorphisms. Let $\epsilon \in \mathcal{U}$. We define numbers $CW_k^{\epsilon}$ for $k \geq 1$ 
by the identity. 
\begin{equation} \label{eq:4.1}
\Delta(f_{\epsilon})(e^X-1) = \sum_{k=1}^{\infty} \frac{CW_k^{\epsilon}}{(k-1)!} X^{k-1}.
\end{equation}
We get then maps
\[
CW_k : \mathcal{U}_{\infty} \rightarrow \Z_p(k).
\]
which are morphisms of $\Z_p[[\Gamma]]$-modules (see \cite[lemma 2.6.2]{CS}). 

We recall that 
\[
\mathcal{P} : \Z_p[[\Z_p]] \rightarrow \Z_p[[T]]
\]
is the Iwasawa isomorphism characterized by $\mathcal{P}([1]) = (1+T)$ and that in our notation if $f \in \Z_p[[T]]$ then $\mu_f := \mathcal{P}^{-1}(f)$. Then we have 
\begin{equation} \label{eq:4.2}
f(e^X-1) = \sum_{k=0}^{\infty} (\int_{\Z_p} x^k d\mu_f(x)) \frac{X^k}{k!}.
\end{equation}

\begin{lemma}  \label{lem:4.3}
Let $f \in \Z_p[[T]]^{\psi = id}$. Then 
\begin{itemize}
\item[(i)] $\mathcal{P}(\mu_f^{\times}) = f(T) - f((1+T)^p-1)$, 
\item[(ii)] $(1-p^k) \int_{\Z_p} x^k d\mu_f(x) = \int_{\Z_p^{\times}} x^k d\mu_f^{\times}(x)$.
\end{itemize}
\end{lemma}

\begin{proof} Let $f \in \Z_p[[T]]^{\psi = id}$. The condition $\psi(f) = f$ is equivalent to the condition $(\varphi \circ \psi)(f) = \varphi(f)$. Hence it is equivalent to the equality
\begin{equation} \label{eq:4.3}
\frac{1}{p} \sum_{\xi \in \mu_p} f(\xi(1+T)-1) = f((1+T)^p-1).
\end{equation}
The power series 
\[
\mathcal{P}(\mu_f^{\times})(T) = f(T) - \frac{1}{p} \sum_{\xi \in \mu_p} f(\xi(1+T)-1).
\]
Hence if $f \in \Z_p[[T]]^{\psi = id}$ then 
\[
\mathcal{P}(\mu_f^{\times})(T) = f(T) - f((1+T)^p-1).
\]
Hence it follows that 
\[
f(e^X-1) - f(e^{pX}-1) = \sum_{k=0}^{\infty} \int_{\Z_p^{\times}} x^k d \mu_f^{\times}(x) \frac{X^k}{k!}.
\]
On the other side we have that 
\[
f(e^X-1) - f(e^{pX}-1) = \sum_{k=0}^{\infty} \int_{\Z_p}x^k d\mu_f(x) \frac{X^k}{k!} - \sum_{k=0}^{\infty} \int_{\Z_p}x^k d\mu_f(x) \frac{p^kX^k}{k!}.
\]
Hence it follows that
\[
\int_{\Z_p^{\times}} x^k d\mu_f^{\times}(x) = (1-p^k) \int_{\Z_p} x^k d \mu_f(x)
\]
for $k \geq 0$. In particular, $\int_{\Z_p^{\times}} d \mu_f^{\times}(x) = 0$.

\end{proof}

We state the last result as  a corollary. 

\begin{corollary} \label{cor:4.4}
Let $f \in \Z_p[[T]]^{\psi = id}$. Then 
\[
\int_{\Z_p^{\times}} d \mu_f^{\times}(x) = 0.
\]
\end{corollary}

\begin{corollary} \label{cor:4.5}
Let $\epsilon \in \mathcal{U}_{\infty}$. Then 
\[
(1-p^k) \int_{\Z_p} x^k d \mu_{\Delta(f_{\epsilon})}(x) = \int_{\Z_p^{\times}} x^k d \mu_{\Delta(f_{\epsilon})}^{\times}(x)
\]
for $k \geq 0$.
\end{corollary}

\section{Group of units} 

Following Ihara (see \cite[Section IV, page 93]{Ihara}) we set 
\[
\mathcal{U}_n^1 := \{ u \in \Z_p[\mu_{p^{n+1}}]^{\times} : N(u)=1 \text{ and } u \equiv 1 (mod \ (\xi_{p^{n+1}}- 1))\},
\]
where $N : \Q_p(\mu_{p^{n+1}}) \rightarrow \Q_p$ is the norm map.  Let 
\[
\mathcal{U}_{\infty}^1 := \varprojlim_n \mathcal{U}_n^1
\]
be the projective limit with respect to the relative norm maps. 


\begin{lem} \label{lem:6.1} We have $\mathcal{U}_{\infty}/\mathcal{U}_{\infty}^1 \cong (\Z/p)^{\times}$.
\end{lem}

\begin{proof} Let  
\[
\mathcal{U}_n' = \{u \in \mathbb{Z}_p[\mu_{p^{n+1}}]^{\times} : u \equiv 1 (\text{mod } (1- \xi_{p^{n+1}})) \}
\]
and 
$\mathcal{U}_{\infty}' = \varprojlim_n \mathcal{U}_n'$. We have $\mathcal{U}_n \cong \mathcal{U}_n' \times (\mathbb{Z}/p\mathbb{Z})^{\times}$ for all $n \geq 1$. Hence it follows that 
\[
\mathcal{U}_{\infty} \cong \mathcal{U}_{\infty}' \times (\mathbb{Z}/p\mathbb{Z})^{\times}.
\]

Next we show that $\mathcal{U}_{\infty}' = \mathcal{U}_{\infty}^1$. Let $u_n \in \mathcal{U}_n'$. 
Then $N(u_n) \equiv 1 (\text{mod } p^n)$ (see \cite[page 113, exercise 4.a]{FV}). 
Let $(u_n) \in \varprojlim_n \mathcal{U}_n'$. Then for all $m \geq n$,  we have $N(u_m) = N(N_{m,n}(u_m)) = N(u_n)$, 
where $N_{m,n}$ is the relative norm map. Hence $N(u_n) \equiv 1 (\text{mod } p^m)$ for all $m \geq n$. 
Therefore $u_n \in \mathcal{U}_n^1$. Hence $\mathcal{U}_{\infty}' = \mathcal{U}_{\infty}^1$ and the result follows.   
\end{proof}

We denote by $E_n$ and $C_n$, the group of units and cyclotomic units, respectively, in $\mathbb{Q}(\mu_{p^n})$. 
Let $\mathbb{Q}(\mu_{p^n})^+$ be the maximal real subfield of $\mathbb{Q}(\mu_{p^n})$. 
We denote by $E_n^+$ and $C_n^+$ the group of units and cyclotomic units in $\mathbb{Q}(\mu_{p^n})^+$ and by 
$\mathcal{E}_n^+$ and $\mathcal{C}_n^+$ the group of $p$-units and cyclotomic $p$-units in $\mathbb{Q}(\mu_{p^n})^+$.

\begin{lem}\label{lem:6.2} 
The inclusions $E_n \hookrightarrow \mathcal{E}_n$ and $E_n^+\hookrightarrow \mathcal{E}_n^+$induce  isomorphisms 
\[
E_n/C_n \xrightarrow{\cong} \mathcal{E}_n/\mathcal{C}_n \;\;{\rm and}\;\;E_n^+/C_n^+ \xrightarrow{\cong} \mathcal{E}_n^+/\mathcal{C}_n^+.
\]
\end{lem}
\begin{prop}\label{prop:6.3}
Let us assume that $p$ does not divide the class number $h(\mathbb{Q}(\mu_{p})^+)$,i.e, that the Vandiver conjecture holds for a prime $p$. Then 
$p$ does not divide $\mid \mathcal{E}_n/\mathcal{C}_n\mid$ for all $n \geq 1$.
\end{prop}
\begin{proof} The Vandiver conjecture for $p$ implies that $p$ does not divide the class number $h(\mathbb{Q}(\mu_{p^n})^+)$ for 
all $n \geq 1$ (\cite[corollary 10.6]{Wa}). Hence $p$ does not divide $|E_n^+/C_n^+|$ for all $n \geq 1$ (\cite[theorem 8.2]{Wa}). 
We have $E_n = \mu_{p^n} E_n^+$ (\cite[theorem 4.12 and corollary 4.13]{Wa}) and $C_n = \mu_{p^n} C_n^+$ (\cite[lemma 8.1]{Wa}). 
Hence it follows that $E_n^+/C_n^+ \cong E_n/C_n$. Therefore it follows from the above lemma that $p$ does not 
divide the order of $\mathcal{E}_n/\mathcal{C}_n$ for all $n \geq 1$. 
\end{proof}

Let $\mathcal{M}_p$ be the maximal abelian pro-$p$ extension of $\mathbb{Q}(\mu_{p^{\infty}})$ unramified outside $p$ and 
let $\mathcal{L}_p$ be the maximal extension of $\mathbb{Q}(\mu_{p^{\infty}})$ contained in $\mathcal{M}_p$ that is unramified. 
Local class field theory defines a canonical  map $\mathbb{Z}_p[[\Gamma]]$-modules
\[
[CL] : \mathcal{U}_{\infty}^1 \rightarrow Gal(\mathcal{M}_p/\mathbb{Q}(\mu_{p^{\infty}})).
\]
The map $[CL]$ induces a surjective morphism of $\mathbb{Z}_p[[\Gamma]]$-modules
\[
[CL] : \mathcal{U}_{\infty}^1 \rightarrow Gal(\mathcal{M}_p/\mathcal{L}_p).
\]
The field $\mathcal{K}_p$ is an abelian pro-$p$ extension of $\mathbb{Q}(\mu_{p^\infty})$ unramified outside $p$. 
Therefore $\mathcal{K}_p \subset \mathcal{M}_p$. Let $\mathcal{K}_p^0$ be the maximal unramified extension of $\mathbb{Q}(\mu_{p^{\infty}})$ contained in $\mathcal{K}_p$. We denote the composition of the map $[CL]$ with the natural map 
\[ 
Gal(\mathcal{M}_p/\mathbb{Q}(\mu_{p^\infty})) \rightarrow Gal(\mathcal{K}_p/\mathbb{Q}(\mu_{p^\infty}))
\]
by $[cl]$. The image of $[cl]$ is the subgroup $Gal(\mathcal{K}_p/\mathcal{K}_p^0)$ of $Gal(\mathcal{K}_p/\mathbb{Q}(\mu_{p^\infty}))$.

The next result we need, is stated without proof in \cite[on page 248]{Ihara1}.
\begin{lem}\label{lem:6.4.}  Let us assume that $p$ does not divide the class number $h(\mathbb{Q}(\mu_{p})^+)$,i.e. 
 that the Vandiver conjecture holds for a prime $p$. Then $\mathcal{K}_p=\mathcal{M}_p^-$.
\end{lem}
\begin{proof} It follows from proposition 3.1,ii) that $\mathcal{K}_p\subset \mathcal{M}_p^-$. Another proof is in \cite[on page 248]{Ihara1}.
First we shall show that $\mathcal{M}_p^-$ is generated over $\mathbb{Q}(\mu_{p^\infty})$ by $p^n$th roots$\Q (\mu _{p^n})^+$ of $p$-units of 
$\mathbb{Q}(\mu_{p^\infty})^+$ for $n\in \NN$.

The extension $\mathbb{Q}(\mu_{p^\infty}) \hookrightarrow \mathcal{M}_p^-$ is pro-$p$, abelian. Hence it is generated by Kummer extensions of the form
$\mathbb{Q}(\mu_{p^n}) \hookrightarrow M=\mathbb{Q}(\mu_{p^n})(a^{1/p^r})$ for some $a \in \Z [\mu _{p^n}]$. We can assume that 
$a\in \Z [\mu _{p^n}]^+$.
 
Let $\fq$ be a prime ideal of $\Z [\mu _{p^n}]^+$ not dividing $p$. If $a \notin \fq$ then the extension $\Q (\mu _{p^n}) \hookrightarrow M$ is 
unramified over $\fq$.
Let us assume that 
$a\in \fq$. Then the extension $M$ of  $\Q (\mu _{p^n})$ is unramified over $\fq$ if and only if $a \in \fq ^{kp^r}$ for 
some $k>0$ and $a\notin \fq ^{kp^r+1}$.
If $\fq$ is a principal ideal generated by $q\in \Z [\mu _{p^n}]^+$ then $a=q^{kp^r}a_1$, $a_1\notin \fq$ and $M=\Q (\mu _{p^n})(a_1^{1/p^r})$.
So let us assume that $\fq$ is not a principal ideal. The fact that $p$ does not divide $h(\mathbb{Q}(\mu_{p})^+)$ implies that $p$ does not divide 
$h(\mathbb{Q}(\mu_{p^n})^+)$ (see \cite[corollary 10.6]{Wa}). Hence there is a positive integer $s$ coprime to $p$ 
such that $\fq^s$ is a principal ideal generated by some $q\in \Z [\mu _{p^n}]^+$. Observe that $a^s =q^{kp^r}a_1$, $a_1\notin \fq$ 
and $\Q (\mu _{p^n})(a_1^{1/p^r})=M$ by the Kummer theory. Therefore $\mathcal{M}_p^-$ is generated over $\Q(\mu _{p^\infty})$ by $p$-units of 
$\Q (\mu _{p^n})^+$ for $n\in \NN$.

It follows form \cite[Corollary 10.6 and Theorem 8.2]{Wa} and from lemma \ref{lem:6.2} that $p$ does not $\Q (\mu _{p^n})^+$divide 
$\mid \mathcal{E}_n^+/\mathcal{C}_n^+ \mid$. Hence it follows from the Kummer theory that $\mathcal{M}_p^{-}$ is generated over $\Q (\mu _{p^\infty})$
by cyclotomic $p$-units of $\Q (\mu _{p^n})^+$ for $n\in \NN$. 
The elements $1-\xi^i_{p^n}$ for $i$ between $0$ and $p^n/2$ and coprime to $p$  and $\xi _{p^n}$ generate $p$-units of $\Q (\mu _{p^n})$. 
The $p$-units of $\Q (\mu _{p^n})^+$ can be expressed by these $p$-units of $\Q (\mu _{p^n})$. Hence it follows that $\mathcal{M}_p^-\subset \mathcal{K}_p$. 
\end{proof}

\section{The main formula} 
We recall the formula from \cite[page 105]{Ihara}
\begin{equation} \label{eq:11}
\int_{\mathbb{Z}_p^{\times}} x^{m-1} d\mathfrak{A}(\widearrow{10})([cl](\epsilon))(x) = (p^{m-1}-1)L_p(m, \omega^{1-m}) CW_m^{\epsilon}, 
\end{equation}
for all odd integers $m > 1$ and all $\epsilon \in \mathcal{U}_{\infty}^1$. 

We shall rewrite the formula in terms of integrals over $\mathbb{Z}_p^{\times}$. It follows from identities in (\ref{eq:4.1}) and (\ref{eq:4.2}) 
and from corollary \ref{cor:4.2} and lemma \ref{lem:4.3} that 
\begin{equation} \label{eq:12}
(1- p^{m-1})CW_m^{\epsilon}= \int_{\mathbb{Z}_p^{\times}} x^{m-1} d\mu_{\Delta(f_{\epsilon})}^{\times}(x)
\end{equation} 
for $m \geq 1$, as by corollary \ref{cor:4.4} both sides are zero for $m=1$. 

Let $c \in \mathbb{Z}_p^{\times} - \mu_{p-1}$. We recall that $\Q (\mu _{p^n})^+$
\[
L_p(1-s , \omega^{\beta}) := 
\frac{1}{\omega^{\beta} (c) \langle c \rangle^s-1} \int_{\mathbb{Z}_p^{\times}} \langle x \rangle^s x^{-1} \omega^{\beta}(x) dE_{1,c}^{\times}(x)
\]
by definition (see \cite[Chapter 4]{Lang}). 

Let $j : \mathbb{Z}_p^{\times} \rightarrow \Z_p^{\times}$ be defined by $j(x) = x^{-1}$. Then $j$ induces $j_* : \Z_p[[\Z_p^{\times}]] \rightarrow \Z_p[[\Z_p^{\times}]]$. Let us set 
\[
\mathcal{E}_c := j_*(E_{1,c}^{\times}). 
\]
Finally we define a measure $\mathcal{F}_c$ on $\Z_p^{\times}$ by 
\[
d\mathcal{F}_c(x) := x d\mathcal{E}_c(x).
\]
It follows from the definition of the measure $\mathcal{F}_c$ that 
\begin{equation} \label{eq:13}
-L_p(m, \omega^{1-m}) = \frac{1}{1-c^{1-m}} \int_{\Z_p^{\times}} x^{-m} dE_{1,c}^{\times}(x) = \frac{1}{1-c^{1-m}} \int_{\Z_p^{\times}} x^{m-1} d \mathcal{F}_c(x)
\end{equation}
for $m\neq 1$. 

Let $c \in \Z_p^{\times}$. Let $\delta_c$ be a measure on $\Z_p^{\times}$ defined by $\int_{\Z_p^{\times}} f(x) d\delta_c(x) = f(c)$ for any continuous function $f$ on $\Z_p^{\times}$. Observe that
\begin{equation} \label{eq:14}
1 - c^{1-m} = \int_{\Z_p^{\times}} x^{m-1} d(\delta_1 - \delta_{c^{-1}})(x).
\end{equation}

\begin{lem} \label{lem:7.1} Let $m\geq 1$ and let $\epsilon \in \mathcal{U}_{\infty}$. Then 
\begin{equation} \label{eq:15}
\int_{\Z_p^{\times}} x^{m-1} d(\delta_1 - \delta_{c^{-1}})(x) \cdot \int_{\Z_p^{\times}} x^{m-1} d \mathfrak{A}(\widearrow{10})^\times ([cl](\epsilon))(x) =
\end{equation}
\[
 \int_{\Z_p^{\times}} x^{m-1} d \mathcal{F}_c(x) \cdot \int_{\Z_p^{\times}} x^{m-1} d \mu_{\Delta(f_{\epsilon})}^{\times} (x)\,.
\]
\end{lem}
\begin{proof} It follows from (\ref{eq:12}), (\ref{eq:13}) and (\ref{eq:14}) that the formula (\ref{eq:11}) can be written in the form
(\ref{eq:15})  for $m > 1$ and  odd. 
Observe that the formula (\ref{eq:15}) holds also for $m =1$ as then both sides vanish. 
Notice that the left hand side of the formula (\ref{eq:15}) vanishes for $m > 0$ and even. On the other side $L_p(m, \omega^{1-m}) = 0$ for $m$ 
even as then $1-m$ is odd. Hence equation (\ref{eq:13}) implies that the right hand side also vanishes for $m > 0$ and $m$ even.  
\end{proof}

\begin{prop} \label{prop:7.2} Let $\epsilon \in \mathcal{U}_{\infty}$ and let $c \in \Z_p^{\times} - \mu_{p-1}$. Then 
\begin{equation} \label{eq:16}
(\delta_1 - \delta_{c^{-1}}) \cdot (\mathfrak{A}(\widearrow{10})^{\times}([cl](\epsilon))) = \mathcal{F}_c \cdot \mu_{\Delta(f_{\epsilon})}^{\times}
\end{equation}
in $\Z_p[[\Z_p^{\times}]]$.
\end{prop}
\begin{proof} The proposition follows immediately from lemma \ref{lem:7.1} and lemma \ref{lem:3.1}.
\end{proof}

Let $S$ be the set of all non-zero divisors in the ring $\Z_p[[\Z_p^{\times}]]$. We set 
\[
\mathfrak{F} : = S^{-1} \Z_p[[Z_p^{\times}]]
\]
for the total ring of fractions. 

\begin{cor} \label{cor:7.3} Let $c \in \Z_p^{\times} - \mu_{p-1}$ and let $\epsilon \in \mathcal{U}_{\infty}^1$. 
Then $[cl](\epsilon) = 0$ in $Gal(\mathcal{K}_p/ \mathbb{Q}(\mu_{p^{\infty}}))$ if and only if 
$(\delta_1 - \delta_{c^{-1}})^{-1} \cdot \mathcal{F}_c \cdot \mu_{\Delta(f_{\epsilon})}^{\times} = 0$ in $\mathfrak{F}$.
\end{cor}
\begin{proof} Following \cite[lemma 4.2.2]{CS} the element $\delta_1 - \delta_{c^{-1}}$ is not a zero divisor in $\Z_p[[\Z_p^{\times}]]$. Hence 
\[
\mathfrak{A}(\widearrow{10})^\times ([cl](\epsilon)) = (\delta_1 - \delta_{c^{-1}})^{-1}\cdot \mathcal{F}_c \cdot \mu_{\Delta(f_{\epsilon})}^{\times}
\]
in $\mathfrak{F}$ by proposition \ref{prop:7.2}. The corollary follows from the fact that the map 
\[
z_p : Gal(\mathcal{K}_p/\mathbb{Q}(\mu_{p^{\infty}})) \rightarrow \Z_p[[\Z_p^{\times}]] 
\]
is injective by proposition \ref{prop:3.1} (i).
\end{proof}
\begin{lem} \label{lem:7.4}
 Let $c,c_1 \in \Z_p ^{\times}-\mu_{p-1}$. Then 
 \[
  (\delta_1 - \delta_ {c^{-1}} )^{-1} \cdot \mathcal{F}_c=(\delta_1 - \delta_{c_1^{-1}})^{-1} \cdot \mathcal{F}_{c_1}
 \]
in the ring $\mathfrak{F}$.
\end{lem}
\begin{proof}
 It follows from the formula (\ref{eq:13}) that
 \begin{equation} \label{eq:16a}
  (1-c_1^{1-m})\int _{\Z_p^\times} x^{m-1}d\mathcal{F}_c (x)=(1-c^{1-m})\int _{\Z_p^\times} x^{m-1}d\mathcal{F}_{c_1} (x)
 \end{equation}
for $m>1$. For $m=1$ both sides of the equality (\ref{eq:16a}) vanish. Hence it follows from lemma \ref{lem:3.1} that
\[
 (\delta_1 - \delta_{c_1^{-1}})  \mathcal{F}_c=(\delta_1 - \delta_{c^{-1}}) \mathcal{F}_{c_1}
\]
in $\Z_p[[\Z_p^\times ]]$. Therefore the lemma follows as $\delta_1 - \delta_{c_1^{-1}}$ and $\delta_1 - \delta_{c^{-1}}$ are not zero divisors in 
$\mathfrak{F}$. 
\end{proof}
Let $f\in \Z _p[[\Z_p^\times ]]$. We denote by $(f)$ an ideal of $\Z _p[[\Z_p^\times ]]$ generated by $f$.
\begin{lem} \label{lem:7.5}
 Let $c\in \Z_p ^{\times}-\mu_{p-1}$.  
 Then $\mathcal{F}_c\in \Z_p[[\Z_p^\times ]]^-$ and $(\mathcal{F}_c)\subset \Z_p[[\Z_p^\times ]]^-$.
\end{lem}
\begin{proof}
 Observe that $L_p(m,\omega ^{1-m})=0$ if $m$ is even as then $1-m$ is odd. Hence it follows from the formula (\ref{eq:13}) that 
 \[
  \int _{\Z_p^\times} x^{m-1}d\mathcal{F}_c (x)=-(1-c^{1-m})L_p(m,\omega ^{1-m})=0
 \]
for $m$ even. Hence it follows that $\Phi (\mathcal{F}_c)\in\big(  \prod_{k=1}^{\infty} \Z_p(k)\big)^{-}$.
Therefore  it follows from  lemma \ref{lem:3.1} that $\mathcal{F}_c\in \Z_p[[Z_p^{\times}]]^-$ 
and also that the ideal $(\mathcal{F}_c )\subset \Z_p[[Z_p^{\times}]]^-$.
\end{proof}

\begin{df} Let $c\in \Z_p^{\times} -\mu _{p-1}$. We set
\[
 \zeta _p:=-(\delta _1-\delta _{c^{-1}})^{-1}\cdot \mathcal{F}_c\,.
\]
\end{df}
It follows from lemma \ref{lem:7.4} that the element $\zeta _p \in \mathfrak{F}$ does not depend on a choice of $c\in \Z_p- \mu _{p-1}$. 
Moreover it follows from \cite[Lemma 4.2.4.]{CS}
that $\zeta _p$ is a pseudo-measure.

\begin{remark} Let $\tilde \zeta _p$ be a pseudo-measure on $\Z_p^\times $ considered in \cite[Proposition 4.2.4.]{CS}. Then the relation between $\zeta _p$ defined in this paper and $\tilde \zeta _p$ is given by $j_*(\tilde \zeta _p)=\zeta _p$.
\end{remark}

The augmentation ideal of $\Z_p[[\Z_p^\times]]$ is the ideal
\[
 I(\Z_p^\times ):=\{\nu \in  \Z_p[[\Z_p^\times ]] \mid \int _{\Z_p^\times }d\nu (x)=0\}\,.
\]

\begin{df} Let $c\in \Z_p^{\times} -\mu _{p-1}$ be such that its class modulo $p^2$ generates $(\Z /p^2\Z )^\times$. We set
\[
 (\zeta _p):=(\mathcal{F}_c)\,.
\]
\end{df}

\begin{lem} \label{lem:7.8}
 The ideal $(\zeta _p)$ is well defined and it is equal $I(\Z_p^\times )\zeta _p$.
\end{lem}
\begin{proof}
The lemma  follows from \cite[Lemma 4.2.5.]{CS}.
\end{proof}

\section{Proof of the  main conjecture assuming the Vandiver conjecture}

Let 
\[
 \mathcal{C}:\mathcal{U}_{\infty} \rightarrow  \Z_p[[\Z_p^\times ]]
\]
be a map defined by $ \mathcal{C}(\epsilon ):=\mu ^\times _{\Delta (f_{\epsilon})}$.
\begin{lem}  \label{lem:8.1}
 The image of the map $ \mathcal{C}:\mathcal{U}_{\infty} \rightarrow  \Z_p[[\Z_p^\times ]]$ is equal $I(\Z_p^\times )$.
\end{lem}
\begin{proof} It follows from corollary \ref{cor:4.2} that the image of $\mathcal{U}_{\infty} $ by the map $[Col]$ is equal to $\Z_p[[T]]^{\psi = id}$.
 Let $f\in \Z_p[[T]]^{\psi = id}$. Then it follows from lemma \ref{lem:4.3} that
 \[
  \mathcal{P}(\mu _f^\times )(0)=\int _{\Z_p^\times }d\mu _f^\times=(f-\varphi (f))(0)=0\,.
 \]
Hence it follows that $\mu ^\times _{\Delta (f_\epsilon)}\in I(\Z_p^\times )$ for $\epsilon \in \mathcal{U}_{\infty}$.

Let $\nu \in I(\Z_p^\times )$. Then $ \mathcal{P}(\nu )\in \Z_p[[T]]^{\psi = 0}$ and $\mathcal{P}(\nu )(0)=0$. 
It follows from \cite[Lemma 2.4.3.]{CS} that there is $g\in \Z_p[[T]]^{\psi = id}$ such that $\mathcal{P}(\nu )=g-\varphi (g)$. 
Moreover $g=[Col](\epsilon)=\Delta (f_\epsilon)$ for some $\epsilon \in \mathcal{U}_{\infty} $ by corollary \ref{cor:4.2}. We have 
\[
 \mathcal{P}(\mu _{\Delta (f_\epsilon)}^\times )=\Delta (f_\epsilon)-\phi (\Delta (f_\epsilon))=g-\phi (g)=\mathcal{P}(\nu)
\]
by lemma \ref{lem:4.3}. Hence it follows that $\nu =\mu _{\Delta (f_\epsilon)}^\times$.
\end{proof}
\begin{prop} \label{prop:8.2} Let $c\in \Z_p-\mu _{p-1}$ be such that its class modulo $p^2$ generates $(\Z /p^2\Z)^\times$. 
Then we have an isomorphism of $\Z_p[[\Gamma ]]$-modules
\[
 Gal (\mathcal{K}_p/\mathcal{K}_p^0)\rightarrow (\zeta _p)
\]
given by $\sigma \mapsto \kappa (\widearrow{10})^\times (\sigma )$.
\end{prop}
\begin{proof}
 Following proposition \ref{prop:3.1} the morphism
 \[
 z_p^{-}: Gal(\mathcal{K}_p/\mathbb{Q}(\mu_{p^{\infty}}))   \rightarrow \mathbb{Z}_p[[\mathbb{Z}_p^{\times}]]
 \]
 is an injective morphism of $\Z_p[[\Gamma ]]$-modules. Following \cite[Lemma 4.2.5.]{CS} 
 \[
  I(\Z_p^\times )=(\delta _1-\delta _{c^{-1}})\mathbb{Z}_p[[\mathbb{Z}_p^{\times}]]\,.
 \]
Hence it follows that
\begin{equation} \label{eq:16b}
(\delta _1-\delta _{c^{-1}})^{-1} I(\Z_p^\times )=\mathbb{Z}_p[[\mathbb{Z}_p^{\times}]]\,.
\end{equation}
The image of the map 
\[
  [cl]:\mathcal{U}^1_\infty\rightarrow Gal(\mathcal{K}_p/\Q (\mu _{p^\infty}))                    
\]
is equal $Gal (\mathcal{K}_p/\mathcal{K}_p^0)$. Hence it follows from proposition \ref{prop:7.2}, lemma \ref{lem:8.1} and the equality (\ref{eq:16b})
that
\[
 z_p^{-}( Gal(\mathcal{K}_p/\mathcal{K}_p^0))=(\zeta _p)\,.
\]
\end{proof}
We recall from Introduction that $\mathcal{L}_p$ is the maximal abelian pro-$p$, contained in $\mathcal{M}_p$, 
extension of $\Q (\mu _{p^\infty })$ unramified everywhere.
Now we shall prove the main conjecture assuming the Vandiver conjecture for a prime $p$.
\begin{prop} \label{prop:8.3} Let us assume that $p$ does not divide the class number $h(\mathbb{Q}(\mu_{p})^+)$,i.e. 
 that the Vandiver conjecture holds for a prime $p$. Then we have an isomorphism of $\Z_p[[\Gamma ]]$-modules
 \[
  Gal(\mathcal{L}_p^-/\mathbb{Q}(\mu_{p^{\infty}}))\cong \Z_p[[\Z_p^\times ]]^-/(\zeta _p)\,.
 \]
\end{prop}
\begin{proof}
 It follows from lemma \ref{lem:6.4.} that $\mathcal{M}_p^-=\mathcal{K}_p$. Therefore $\mathcal{L}_p^-=\mathcal{K}_p^0$. The morphism
 \[
  z_p^{-}: Gal(\mathcal{K}_p/\mathbb{Q}(\mu_{p^{\infty}}))   \rightarrow \mathbb{Z}_p[[\mathbb{Z}_p^{\times}]]
 \]
is an isomorphism of $\Z_p[[\Gamma ]]$-modules by proposition  \ref{prop:8.2} and by lemma \ref{lem:3.3}. Hence it follows from proposition \ref{prop:8.2}
that $Gal(\mathcal{K}_p^0/\mathbb{Q}(\mu_{p^{\infty}}))\cong \Z_p[[\Z_p^\times ]]^-/(\zeta _p)$
\end{proof}

\section{The Ihara formula for all $m$}

In this section we show that the Ihara formula from Introduction holds for all $m$. In fact only the case $m=1$ requires a careful proof.
\begin{lem}  \label{lem:9.1} Let $c\in \Z_p^{\times} - \mu_{p-1}$. Then
\[
 \int_{\Z_p^{\times}}d\mathcal{F}_c(x)=(1-p^{-1})\log (c^{-1}).
\]
\end{lem}
\begin{proof}  We recall that $E_{1,c}$ is the regularized Bernoulli measure and $E_{1,c}^{\times}$ is its restriction to $ \Z_p^{\times}$. The power series 
$\mathcal{P}(E_{1,c})$ is equal $1/T-c/((1+T^c-1)$ (see \cite[Chapter 4, Proposition 3.4.]{Lang}).Hence it follows from lemma  \ref{lem:4.3} that
\[
 \mathcal{P}(E_{1,c}^{\times})(T)=\frac{1}{T}-\frac{c}{(1+T)^c-1}-\frac{1}{(1+T)^p-1}+\frac{c}{(1+T)^{pc}-1}\,.
\]
Observe that $\varepsilon =\big( \frac{\xi _{p^{n+1}}-1}{{\xi _{p^{n+1}}^c-1}}\big)_{n\in \NN}\in \mathcal{U}_{\infty}$. The corresponding power series 
$f_{\varepsilon}=\frac{T}{(1+T)^c-1}$ belongs to $W$.
Let us set
\[
 g(T):=\log (\frac{T}{(1+T)^c-1})-\frac{1}{p} \log(\frac{(1+T)^p-1}{(1+T)^{pc}-1})\,.
\]
It follows from \cite[lemma  2.5.1]{CS} that $g(T)\in \Z_p[[T]]^{\Psi =0}$. The operator $D$ defined by $(Df)(T):=(1+T)f^\prime (T)$
is an automorphism of  $\Z_p[[T]]^{\Psi =0}$ (see \cite[Corollary on page 2]{Coleman1}). One checks that
\[
 Dg=\mathcal{P}(E_{1,c}^\times)\,.
\]
Therefore
\[
 \int_{\Z_p^{\times}}d\mathcal{F}_c(x)=\int_{\Z_p^{\times}}x^{-1}dE_{1,c}^\times(x)=(D^{-1}\mathcal{P}(E_{1,c}^\times))(0)=g(0)=(1-p^{-1})\log (c^{-1})
\]
(see \cite[Lemma 3.4.]{NSW}).
\end{proof}

We define a sequence of $\Z_p$-algebra isomorphisms
\[
 \Z_p[[\Z_p^\times ]]\rightarrow \Z_p[[\mu _{p-1}\times (1+p\Z_p)  ]],\, 
 \Z_p^\times \ni [x]\mapsto [\omega (x),x\omega(x)^{-1}]\in \Z_P[[\mu _{p-1}\times (1+p\Z_p)]];
\]
\[
 \Z_p[[\mu _{p-1}\times (1+\Z_p)]]\rightarrow \Z_p[\Delta ][[1+p\Z_p]],\,[\varepsilon , x]\mapsto \varepsilon [x],
\]
where we view $\mu _{p-1}$ as an abstract group denoted by $\Delta$;
\[
 \Z_p[\Delta ][[1+p\Z_p]]\rightarrow \Z_p[\Delta ][[\Z_p]],\, [x]\mapsto [\frac{\log x}{\log q}],
\]
where $q=p+1$ if $p\neq 2$ and $q=5$ if $p=2$.
The composition of these isomorphisms we denote by $\alpha$. Let 
\[
 \varepsilon : \Z_p[\Delta ][[\Z_p]]\rightarrow \Z_p[[\Z_p]]
\]
be the augmentation map with respect to $\Delta$. Let us set
\[
 A:=\varepsilon \circ \alpha \, .
\]
Then $A$ is also a morphism of $\Z_p$-algebras.

\begin{lem}  \label{lem:9.2}
Let $\mu \in  \Z_p[[\Z_p^\times ]]$. Then
\[
 \int_{\Z_p^{\times}}d\mu =\int_{\Z_p }dA(\mu )\,.
\]
\end{lem}
\begin{proof}
 One cheks that the formula is true for any $\mu \in \Z_p[\Z_p^\times]$. Hence by continuity the formula holds for any $\mu \in \Z_p[[\Z_p^\times ]]$.
\end{proof}
\begin{lem}  \label{lem:9.3} 
Let $\mu \in  \Z_p[[\Z_p^\times ]]$ be such that $\int_{\Z_p^{\times}}d\mu =0$. Let $c\in \Z_p^\times $ be such that its class modulo $p^2$ 
generates $(\Z/p^2\Z)^\times$. Then there exists a unique $\nu \in \Z_p[[\Z_p^\times ]]$ such that
 \[
  \mu =(\delta _1 -\delta _c)\cdot \nu
 \]
and
\[
 \int_{\Z_p^{\times}}d\nu =(-\frac{\log q}{\log c})\int_{\Z_p }xdA(\mu )(x)\,.
\]
\end{lem}
\begin{proof}
 It follows from \cite[Lemma 4.2.5]{CS} and from the fact that $\delta _1-\delta _c$ is not a zero divisor (see \cite[proof of Lemma 4.2.2.]{CS})
 that there is a unique $\nu \in  \Z_p[[\Z_p^\times ]]$ such that 
 \[
  \mu =(\delta _1-\delta _c)\cdot \nu
 \]
in $ \Z_p[[\Z_p^\times ]]$. Applying the $\Z_p$-algebra homomorphism $A$ to the above formula we get
\[
 A(\mu )=(\delta _0-\delta _{x_0})\cdot A(\nu ) 
\]
in $\Z_p[[\Z_p]]$, where $x_0=\frac{\log c}{\log q}$. Hence we have the following equality of power series in $\Z_p[[T]]$
\[
 \mathcal{P}(A(\mu ))(T)=(1-(1+T)^{x_0})\cdot \mathcal{P}(A(\nu ))(T)\,.
\]
It follows from lemma \ref{lem:9.2} that
\[
 \mathcal{P}(A(\mu))(0)=\int _{\Z_p}dA(\mu )=\int_{\Z_p^{\times}}d\mu =0\,.
\]
Comparing the coefficients at $T$ we get 
\[
 \int _{\Z_p}xdA(\mu )(x)=(-x_0)\int _{\Z_p}dA(\nu )=(-x_0)\int_{\Z_p^{\times}}d\nu \,.
\]
\end{proof}

In the next proposition we present the analogue of the Ihara formula from Introduction for $m=1$. 
Notice that $1-p^{-1} $ appearing in our formula is the residue of the p-adic zeta function of Kubota-Leopoldt.

\begin{prop}  \label{prop:9.4}
 Let $\epsilon \in \mathcal{U}_{\infty}^1$. Then we have
 \[
 \kappa (p)([cl](\epsilon ))= \int_{\Z_p^{\times}} d\mathfrak{A}(\widearrow{10})([cl](\epsilon ))=(1-p^{-1})\big(-\log q\int_{\Z_p  }xdA(\mu ^\times _{\Delta (f_{\epsilon})}) (x)\big)\,.
 \]
\end{prop}
\begin{proof}
 Let $c\in \Z_p^\times $ be such that its class modulo $p^2$ generates $(\Z/p^2\Z)^\times$. 
 Following corollary \ref{cor:4.5} $\int_{\Z_p^{\times}} d\mu ^\times _{\Delta (f_{\epsilon})}=0$.
 Hence it follows from lemma \ref{lem:9.3} that there is a unique $\nu _{\epsilon , c}\in \Z_p[[\Z_p^\times ]]$ such that
 \[
  \mu ^\times _{\Delta (f_{\epsilon})}=(\delta _1-\delta _{c^{-1}})\cdot \nu _{\epsilon , c}\,.
 \]
Hence it follows from proposition \ref{prop:7.2} that 
\[
 \mathfrak{A}(\widearrow{10})^\times([cl](\epsilon )) =  \mathcal{F}_c  \cdot \nu _{\epsilon , c}\,.
\]
It follows from lemma \ref{lem:3.1} that 
\[
 \int_{\Z_p^{\times}} d\mathfrak{A}(\widearrow{10})^\times([cl](\epsilon ))  = 
 \int_{\Z_p^{\times}} d\mathcal{F}_c  \cdot  \int_{\Z_p^{\times}} d\nu _{\epsilon , c}\,.
\]
Let us observe that $ \int_{\Z_p^{\times}} d\mathfrak{A}(\widearrow{10})(\sigma )^\times =\kappa (p)(\sigma )$ for any $\sigma \in G_{\Q (\mu _{p^\infty })}$.
On the other side
\[
 \int_{\Z_p^{\times}} d\mathcal{F}_c  \cdot  \int_{\Z_p^{\times}} d\nu _{\epsilon , c}=(1-p^{-1})\log (c^{-1})\big( -\frac{\log q}{\log (c^{-1})}
 \int_{\Z_p  }xdA(\mu ^\times _{\Delta (f_{\epsilon})}) (x)\big)=
\]
\[
 (1-p^{-1}) \big( - {\log q} \int_{\Z_p  }xdA(\mu ^\times _{\Delta (f_{\epsilon})}) (x)\big)
\]
by lemmas \ref{lem:9.1} and \ref{lem:9.3}.
\end{proof}
Let us define 
\[
 \Psi : \Z_p[[\Z_p^\times ]]\rightarrow \prod_{k=-\infty}^{\infty} \Z_p(k)
\]
by the formula
\[
 \Psi (\mu )=\big( \int_{\Z_p^{\times}}x^{k-1} d\mu (x)\big)^\infty _{k=-\infty}\,.
\]
As in the proof of lemma \ref{lem:3.1} one checks that $\Psi (\mu \cdot \nu )=\Psi (\mu )\Psi (\nu )$.
Hence it follows from proposition \ref{prop:7.2} that
\[
 (1-c^{1-k})\big( \int_{\Z_p^{\times}}x^{k-1}d\mathfrak{A}(\widearrow{10})^\times ([cl](\epsilon ) )   \big)=
 \big( \int_{\Z_p^{\times}} x^{k-1}d\mathcal{F}_c (x)\big) \big( \int_{\Z_p^{\times}} x^{k-1}d\mu ^\times _{\Delta (f_{\epsilon})}) (x)\big)\,.
\]
Following (\ref{eq:13}), $\int_{\Z_p^{\times}} x^{k-1}d\mathcal{F}_c (x)=-(1-c^{1-k})L_p(k,\omega ^{1-k})$ for $k\neq 1$. 
Hence we have proved the following result.
\begin{prop} \label{prop:9.6}
 Let $\epsilon \in \mathcal{U}_{\infty}^1$. For $k\neq 1$ we have
 \[
  \int_{\Z_p^{\times}}x^{k-1}d\mathfrak{A}(\widearrow{10})^\times ([cl](\epsilon ) )    =
  -L_p(k,\omega ^{1-k})\int_{\Z_p^{\times}} x^{k-1}d\mu ^\times _{\Delta (f_{\epsilon})}  (x)\,.
 \]
\end{prop}

{\bf Acknowledgements} The second author would like to thank very much Hiroaki Nakamura for discussions.

\end{document}